\documentclass[12pt]{amsart}
\usepackage{amssymb}
\usepackage{amsfonts}
\usepackage{amscd}
\usepackage[T1]{fontenc}

\usepackage{xy}
\input xy
\xyoption{all}
\usepackage{pdflscape}
\usepackage{hyperref}

\usepackage[english]{babel}

\def\A{\mathbb{A}}

\def\B{\mathbb{B}}

\def\P{\mathbb{P}}

\let\myacute=\'

\def\<{\langle}
\def\>{\rangle}

\def\cL{\mathcal{L}}

\def \begindm {\begin{displaymath}}

\def \enddm {\end{displaymath}}

\def\A{\mathbb{A}}
\def\Q{\mathbb{Q}}

\def\cL{\mathcal{L}}

\def\cM{\mathcal M}
\def\cO{\mathcal{O}}
\def\cA{\mathcal{A}}

\newtheorem{thm}{Theorem}[section]
\newtheorem{lemma}{Lemma}[section]
\newtheorem{cor}{Corollary}[section]

\newtheorem{Def}{Definition}[section]

\numberwithin{equation}{section}

\long\def\symbolfootnote[#1]#2{\begingroup\def\thefootnote{\fnsymbol{footnote}}\footnote[#1]{#2}\endgroup}

\title[Axioms for Restricted Products]{Axioms for Commutative Unital Rings elementarily Equivalent to Restricted Products of Connected Rings}
\author[J. Derakhshan]{Jamshid Derakhshan}
\address{St Hilda's College, University of Oxford, Cowley Place, Oxford OX4 1DY, UK}
\email{derakhsh@maths.ox.ac.uk}

\author[A. Macintyre]{Angus Macintyre${}^{\dag}$}
\address{School of Mathematical Sciences, Queen Mary, University of London, Mile End Road, London E1 4NS, and 
School of Mathematics, University of Edinburgh, James Clerk Maxwell Building, Peter Guthrie Tait Road
Edinburgh EH9 3FD, UK}
\email{angus@eecs.qmul.ac.uk}

\thanks{${}^{\dag}$Supported by a Leverhulme Emeritus Fellowship}

\begin{document}


    
\begin{abstract} 
We give axioms in the language of rings augmented by a 1-ary predicate symbol $Fin(x)$ with intended interpretation 
in the Boolean algebra of idempotents as the ideal of finite elements, i.e. finite unions of atoms. We prove that any 
commutative unital ring satisfying these axioms is elementarily equivalent to a restricted product of 
connected rings. This is an extension of the results in \cite{elem-prod} for products. While the results in \cite{elem-prod} give a converse to the Feferman-Vaught theorem for products, our results prove the same for restricted products. We give a complete set of axioms in the language of rings for the ring of adeles of a number field, uniformly in the number field.

\end{abstract}

\maketitle

\section{\bf Introduction}\label{sec-introduction}

This paper is a natural sequel to \cite{elem-prod} and the main results and proofs are natural extensions of those in \cite{elem-prod}. In many cases we will simply refer to the material from \cite{elem-prod}

\cite{elem-prod} deals with the model theory of products of connected unital rings, and can be construed as providing a partial converse to the Feferman-Vaught Theorem \cite{FV} in the special case of products $\prod_{i\in I} R_i, i\in I,$ where $R_i$ are connected commutative unital rings and $I$ an index set (Recall that a commutative ring $R$ is connected if $0,1$ are the only idempotents of $R$).  The converse concerns the issue of providing axioms for {\it rings elementarily equivalent to rings $\prod_{i\in I} R_i$ as above}.  The solution of this problem is given in \cite{elem-prod} and, inter alia has applications to non-standard models of PA (first order Peano arithmetic) in \cite{PDAJM}.

In this paper we start with rings $\prod_{i\in I} R_i, i\in I,$ as above, but work with certain subrings, namely restricted products with respect to a formula $\varphi(x)$ of the language of rings (in a single variable $x$), 
defined as the set of all $f\in \prod_{i\in I} R_i$ so that 
$\{i: R_i \models \varphi(f(i))\}$ is cofinite. Provided that $\varphi(x)$ defines a unital subring of each $R_i$, the above subset is in fact a subring of $\prod_{i\in I} R_i$ (not in general definable). 

We obtain, for restricted products, results exactly analogous to those of \cite{elem-prod} for products. 
Given $\varphi(x)$, we provide axioms in the language of rings augmented by a predicate $Fin(x)$, and prove that
any commutative unital ring satisfying these axioms is elementarily equivalent to a restricted product, with respect to 
$\varphi(x)$, of connected rings. The standard interpretation of $Fin(x)$ in any Boolean algebra, and in particular in the Boolean algebra of idempotents of $R$ is the ideal of finite elements, i.e. finite unions of atoms.

The canonical example of a such a restricted product, very important in number theory (see Cassels and Frohlich \cite{CF}), is $\A_K$, the ring of adeles over a number field $K$. Here $I$ is the set of normalized absolute values $v$ on $K$ up to equivalence, $R_i$ is the completion $K_v$ of $K$ at $v$, and $\varphi(x)$ is a formula of the language of rings that defines, uniformly for all $v$, the valuation ring $\cO_v$ of $K_v$. (That there is such a $\varphi(x)$ is nontrivial, and it is an important result that there is an $\exists \forall$-formula $\varphi(x)$ that works uniformly for all $\cO_v$, and hence for all adele rings  uniformly in $K$, see \cite{CDLM} and \cite{DM-ad}).

In the case of adeles $\A_K$, the set of idempotents with finite support is definable by a formula of the language of rings independently of $K$ (cf. \cite{DM-ad}, and a new proof given at the end of this paper in a ring-theoretic situation). Thus we can derive axioms in the language of rings for the adeles, uniformly in $K$.

\section{\bf The Boolean algebra of idempotents of a ring}

We shall denote the language of rings by $\cL_{rings}=\{+,.,0,1\}$ and the language of Boolean algebras by 
$\mathcal{L}_{Boolean}=\{\vee,\wedge,\neg,0,1\}$. We start by recalling the definition of a restricted product of structures with respect to a formula (cf. \cite{DM-ad}, \cite{DM-supp}, \cite{DM-ad2}).

\begin{Def}\label{def-rest} Let $\cL$ be a language and $(\cM_i)_{i\in I}$ a family of $\cL$-structures. Let $\varphi(x)$ be a 
$\cL_{rings}$-formula in the single variable $x$. The restricted direct product of $\cM_i$ with respect to $\varphi$ (also called product restricted by $\varphi$) is the subset of the product $\prod_{i\in I} \cM_i$ consisting of all $f$ such that $\cM_i \models \varphi(f(i))$ for all but finitely many $i\in I$.\end{Def}

We denote this restricted product by $\prod_{i\in I}^{(\varphi)} \cM_i$. It is a substructure of the generalized product defined by Feferman and Vaught in \cite{FV} provided that $\varphi(\cM_i)$ is a substructure of $\cM_i$ for all $i\in I$. The results in \cite{FV} and \cite{DM-ad}, \cite{DM-ad2}, \cite{DM-supp} yield general quantifier eliminations for such restricted products, where $\cL$ is any many-sorted language. One can deduce, among other results, quantifier elimination for adeles, and results on definable subsets of adeles and their measures.

%From now on we deal with a converse situation, and give axioms for such restricted products.

\

\subsection{Atoms and Stalks}

\

\

We follow as much as possible the development from Section 1 of \cite{elem-prod}.

\begin{Def} Let $R$ be a commutative unital ring. The set $\{x: x=x^2\}$ of idempotents
is a Boolean algebra, denoted by $\B$, with operations  
$$e \wedge f=ef,$$
$$\neg e=1-e,$$
$$0=0,$$
$$1=1,$$
$$e\vee f=1-(1-e)(1-f)=e+f-ef.$$
It carries a partial ordering defined by 
$e\leq f\Leftrightarrow ef=e$ (which is $\cL_{rings}$-definable). The {\it atoms} of $\B$ are the minimal idempotents (with respect to the ordering) that are not equal to $0,1$. 
(In fact we {\it assume} $0\neq 1$).\end{Def}

Note that if $R$ is a product, over an index set $I$, then $\B$ is isomorphic to the Boolean algebra of subsets of $I$ via 
"characteristic functions".

\begin{lemma}\label{lem1} For any $e$ in $\B$ we have $R/(1-e)R \cong eR \cong R_e$, where $R_e$ is the localization of $R$ at $\{e^n: n\geq 0\}$.
\end{lemma}
\begin{proof} See Lemma 1 in \cite{elem-prod}.\end{proof}

We call $R_e$ the stalk of $R$ at $e$. Of special important are the $R_e$ for {\it atoms} $e$.

\

Now we gradually impose axioms on $R$, in order to get a converse to Feferman-Vaught for restricted products.

\

{\bf Axiom 1.} $\B$ is atomic.

\

{\bf Notes.} This holds if $R$ is a restricted product of {\it connected} rings. One does not even need the restricting formula $\varphi(x)$ to be definable. Moreover $R$ and the unrestricted product have the same idempotents. The basic example is $\A_K$ embedded in $\prod_{v} K_v$.

\

Now we go through a series of consequences of the current axioms, and additions of new axioms.

\begin{lemma}\label{lem2} If $f\in \B$, and $f\neq 0$, then $f=\bigvee\{e: \text{e ~ an~ atom}, ~ e\leq f\}$, 
(where $\bigvee$ is union or supremum).
\end{lemma}
\begin{proof} This is Lemma 2 in \cite{elem-prod}.\end{proof}

We turn to Boolean values and follow 1.3 of \cite{elem-prod}.

\begin{Def} Let $\Theta(x_1,\dots,x_n)$ be a formula of the language of rings, and $f_1,\dots,f_n\in R$. Then 
$[[\Theta(f_1,\dots,f_n)]]$ is defined as 
$$\bigvee_{e} \{e: e~\text{an~atom},~R_e\models \Theta((f_1)_e,\dots,(f_n)_e)\}$$
provided $\bigvee$ exists in $\B$. Here $(f)_e$ is the natural image of $f$ in $R_e$ (or, seen from perspective of Lemma \ref{lem1}, $f+(1-e)R$).\end{Def}

\

{\bf Axiom 2.} $[[\Theta(f_1,\dots,f_n)]]$ exists (as an element of $\B$).

\

{\bf Notes.} If $R$ is a product of structures then $\B$ is complete, however completeness of a Boolean algebra 
is not a first-order property. 

Axiom 2 is a substitute for completeness (and follows from it). 

Axiom 2 is true in a restricted product of connected rings with respect to a given formula $\varphi(\bar x)$.

\

\subsection{Boolean Values and Patching}

\

\

The $[[\Phi(f_1,\dots,f_n)]]$ are in $\B$, and occur in \cite{FV} in the context of products, 
with a different notation. The $[[...]]$ notation comes from Boolean valued model theory \cite{}.

The next Lemmas come from 1.4 of \cite{elem-prod}.

\begin{lemma}\label{lem3-5} Let $\Theta_1,\Theta_2,\Theta_3$ be $\cL_{rings}$-formulas in the variables $x_1,\dots,x_n$. 
Then for any $f_1,\dots,f_n \in R$, \begin{itemize}
\item $[[(\Theta_1 \wedge \Theta_2)(f_1,\dots,f_n)]]=[[\Theta_1(f_1,\dots,f_n)]]\wedge [[\Theta_2(f_1,\dots,f_n)]]$,
%\begin{lemma}\label{lem4}
\item $[[(\neg \Theta)(f_1,\dots,f_n)]]=\neg [[\Theta(f_1,\dots,f_n)]],$
%\end{lemma}
%\begin{lemma}\label{lem5} 
\item $[[(\Theta_1 \vee \Theta_2)(f_1,\dots,f_n)]]=[[\Theta_1(f_1,\dots,f_n)]] \vee [[\Theta_2(f_1,\dots,f_n)]]$.
\end{itemize}\end{lemma}
\begin{proof} These statements are Lemmas 3-5 in \cite{elem-prod}
\end{proof}

These are some of the ingredients used in inductive proofs of result in \cite{FV}.

\

We add another axiom, taken from 1.4. of \cite{elem-prod}

\

{\bf Axiom 3.} For any atomic formula $\Theta(x_1,\dots,x_n)$ of the language of rings,
$$R\models \Theta(f_1,\dots,f_n) \Leftrightarrow \B\models [[\Theta(f_1,\dots,f_n)]]=1.$$

This is evidently true in restricted products, no matter what $\varphi(x)$ is. 

Now we {\it fix} a $\varphi(x)$, and aim for axioms true in restricted direct products $R$ with respect to $\varphi(x)$.

We come now to a fundamental point. Classically the notion of restricted product appeals to the 
absolute notion of \underline{finite} which is not, of course, first-order. We are aiming for first-order axioms in some natural formalism. As already suggested, we are going to use an idea from Feferman -Vaught \cite{FV} of working with Boolean algebras $\B$ with a distinguished subset $\mathcal{F}in$, which in the case of the power set algebra is the ideal of finite sets. In the case of a Boolean algebra of idempotents, $\mathcal{F}in$ will be the set of finite idempotents, as explained earlier. 

We will shortly be concerned with other interpretations of a predicate symbol for $\mathcal{F}in$, indispensable for understanding nonstandard models of our axioms (and in particular nonstandard models of the theory of the adeles.

But first we use provisional "axioms" where "finite" really means finite, and "finite idempotents " really mean finite idempotents, and "cofinite" really means cofinite.

We could avoid this ,and pass directly to the general case. But we prefer to discuss a provisional axiom connected to the kind of patching used in \cite{FV}

\

{\bf Axiom $4^+$.} For all $\Theta(x_1,\dots,x_n,w)$, $f_1,\dots,f_n$, there is a $g\in R$ such that if 
$$[[\exists w (\varphi(w) \wedge \Theta(f_1,\dots,f_n,w))]]$$ 
is cofinite in 
$$[[\exists w \Theta(f_1,\dots,f_n,w)]],$$ 
then $[[\exists w \Theta(f_1,\dots,f_n,w)]]$ is cofinite in 
$[[\Theta(f_1,\dots,f_n,g)]]$.

This is clearly true in restricted products with respect to $\varphi(x)$ (use Axiom of Choice).

\

{\bf Note.}  \cite{elem-prod} has a simpler Axiom 4 for the unrestricted product case. That is not needed here.

\

{\bf Note.} From now on, we will get involved with not only $\B$, but with the ideal $\mathcal{F}in$ in $\B$ consisting 
of finite elements of $\B$, i.e. finite unions of atoms.

\

We have to enrich the first-order language of Boolean algebras by a 1-ary predicate symbol 
$Fin(x)$. For our purposes $\B$ 
will be atomic as above, and $Fin(x)$ interpreted as 
the ideal of finite support idempotents. The interpretation of $Fin(x)$ in a Boolean algebra of sets, e.g. the powerset $\P(I)$ of a set $I$ is the (Boolean) ideal of finite sets. 

However, note that Axiom $4^+$ is {\it not} first-order. 
Any anxieties about this should be removed by considering the result that the theory of the class of all
infinite atomic Boolean algebras in the language of Boolean algebras augmented by $Fin(x)$ is axiomatizable and 
complete (and admits quantifier elimination). This is proved first by Tarski but we give a new proof with explicit axioms in 
\cite{DM-bool}. See also Section \ref{sec-bool} below. \cite{DM-bool} contains a unified treatment that includes further expansions by predicates for "congruence conditions on cardinality of finite sets". 

We return to this matter later, reformulating Axiom $4^+$ in terms of $Fin(x)$.

\

\subsection{Partitions}

\

\

In order to sketch a proof of a useful generalization of \cite{FV} to our more restrictive situation (rings $R$ satisfying the axioms listed above) we need to review several notions of partition used in \cite{FV}.

\

{\bf Notion 1.} In a Boolean algebra $B$ a partition is a finite sequence $<Y_1,\dots,Y_m>$ of elements of $B$ such that
$$Y_1 \vee \dots \vee Y_m=1$$
and $Y_i \wedge Y_j =0$ if $i\neq j$. (We do not insist that each $Y_i\neq 0$, but do insist that the sequence is finite).

\
We note that in the definition of partition "finite" will always mean finite.

\

{\bf Notion 2.} For a first-order language $L$, a fixed $m$, and $L$-formulas 
$$\Theta_1(x_1,\dots,x_m),\dots,\Theta_m(x_1,\dots,x_m)$$
the sequence $<\Theta_1,\dots,\Theta_m>$ is a {\it partition} if the formulas 
$\Theta_1 \vee \dots \vee \Theta_m$ and $\neg(\Theta_i \wedge \Theta_j)$ (where $i\neq j$) are logically valid.

(This is of course ultimately a special case of Notion 1).

The basic lemmas about disjunctive normal form in propositional calculus, when applied to formulas 
$$\psi_1(x_1,\dots,x_m),\dots,\psi_l(x_1,\dots,x_m)$$
give constructively a partition whose elements are propositional combinations of the $\psi_i$'s. This is used crucially in \cite{FV}. 

The final result we need before sketching \cite{FV} for all our rings is.

\begin{lemma}\label{lem6+} (Analogue of Lemma 6 in \cite{elem-prod}) Suppose $Y_1,\dots,Y_k$ is a partition of $\B$. 
Suppose the sequence
$$<\Theta_1(x_1,\dots,x_m,x_{m+1}),\dots,\Theta_k(x_1,\dots,x_m,x_{m+1})>$$ 
is a partition. Suppose $f_1,\dots,f_m\in R$ and 
$$Y_j\leq [[\exists x_{m+1} \Theta_j(f_1,\dots,f_m,x_{m+1})]]$$
for each $j$. Suppose in addition that for each $j$
$$Y_j \wedge \neg [[\exists x_{m+1} \varphi(x_{m+1}) \wedge \Theta_j(f_1,\dots,f_m,x_{m+1})]]$$
is finite. Then there is a $g$ in $R$ so that
$$Y_j\subseteq [[\Theta_j(f_1,\dots,f_m,g)]]$$
for all $j$.\end{lemma}
\begin{proof} Apply Axiom $4^+$ to each $Y_j$ and $\Theta_j$, $f_1,\dots,f_m$ to get $g_j\in R$ so that 
$Y_j\subseteq [[\Theta_j(f_1,\dots,f_m,g_j)]]$. Now let $g=\sum_j g_j.e_j$, where $e_j$ is the idempotent corresponding to 
$Y_j$.
\end{proof}

\

{\bf Note.} Later we re-do this in terms of $Fin$ (subject to Tarski's axioms).

\

\subsection{The Augmented Boolean Formalism}\label{sec-bool}

\

\

We adjoin to the first-order language for Boolean algebras $\cL_{Boolean}$ a unary predicate symbol $Fin(x)$. Denote the augmented language by $\cL_{Boolean}^{fin}$. The standard interpretation of $Fin(x)$ in a Boolean algebra is the ideal $\mathcal{F}in$ 
of finite elements, i.e. finite unions of atoms. 
%Since we deal, in this paper, only with atomic $\B$, $\mahcal{F}in$ is simply the ideal of finite elements. 
However, the class of such $(\B,\mathcal{F}in)$ is not elementary, and it is important for us to give a computable set of axioms complete up to specifying the the number of atoms below an element. We do this in \cite{DM-bool} (as part of a new expansion of $\cL_{Boolean}$), but the original work was done by Tarski (see \cite{FV}). Here are the essential points. 

Let $T$ be the theory of infinite atomic Boolean algebras in the language 
$\mathcal{L}_{Boolean}$ expanded by the definable relations $C_k(x)$ ($k=1,2,\dots$) 
with the interpretation that $x$ has at least $k$ atoms $\alpha \leq x$. Tarski proved that in this language the 
theory of infinite atomic Boolean algebras is complete, admits quantifier elimination, and is decidable (see \cite{FV}, \cite{DM-bool}). The axioms for this theory state that the models are infinite Boolean algebras and every nonzero element has an atom below it. Let $\sharp(x)$ denote the number of atoms $\alpha$ such that $\alpha \leq x$. 

Now we further expand the given language by adding the predicate $Fin(x)$ with the above interpretation in any Boolean algebra, and obtain $\cL_{Boolean}^{fin}$. We add to the axioms of $T$ the axioms stating that $\mathcal{F}in$ is a proper ideal, the sentence
$$
\forall x (\neg Fin(x) \Rightarrow (\exists y)(y<x \wedge \neg Fin(y) \wedge \neg Fin(x-y))).
$$
and, for each $n<\omega$, the sentence $\forall x (\sharp(x)\leq n \Rightarrow Fin(x))$.
This defines an $\cL_{Boolean}^{fin}$-theory $T^{fin}$.

\begin{thm}\label{bool1}\cite{FV},\cite{DM-bool} The theory $T^{fin}$ of infinite atomic Boolean algebras with the 
set of finite sets distinguished is complete, decidable and 
has quantifier elimination with respect to all the $C_n$, $(n\geq 1)$, and $Fin$ (i.e. in the language 
$\mathcal{L}_{Boolean}^{fin}$). The axioms required for completeness are the 
axioms of $T$ together with 
sentences expressing that $Fin$ is a proper ideal, the sentence
$$
\forall x (\neg Fin(x) \Rightarrow (\exists y)(y<x \wedge \neg Fin(y) \wedge \neg Fin(x-y))).
$$
and, for each $n<\omega$, the sentence $\forall x (\sharp(x)\leq n \Rightarrow Fin(x))$.
\end{thm}

%\begin{thm} \cite{DM-bool},\cite{FV}\begin{itemize}
%\item The theory $T^{fin}$ together with the sentence $Fin$ is proper is complete, and if $\B$ is infinite and 
%$Fin$ the ideal of finite sets, $(\B,Fin)\models T^{fin}$.
%\item For $k=0,1,2,\dots$ each of the theories $T^{fin}\cup \{C_k(1),\neg C_{k+1}(1)\}$ is complete, with models of %cardinality
%$2^k$, and $Fin=\B$.
%\item (i) and (ii) give all complete extensions of $T^{fin}$.\end{itemize}\end{thm}

{\bf Note:} This is important for measurability of definable sets in adele rings as in \cite{DM-ad}.

\

{\bf Note:} In \cite{DM-bool} we prove that $\mathcal{F}in$ is not definable in the language of the theory $T$.

\

\subsection{Modifying the Preceding (provisional) Axioms for $R$}

\

\

Our axioms given so far are not first-order. To rectify this, we first do the following. 
Work with rings $R$ together with a distinguished ideal in $\B$. In the adelic cases this ideal will be the ideal of finite idempotents, 
but we will also be interested in other ideals. We write $\mathcal{F}in$ for the distinguished ideal. We modify axiom $4^+$ to Axiom $4^{Fin}$ (which is still not first-order). 
We let $\mathcal{F}in$ be the set of realizations in $R$ of the predicate $Fin(x)$.

%We aim for first-order axioms in $R$. So we simply modify Axiom $4^+$ to 

\

{\bf Axiom $4^{fin}$}. There is an ideal $\mathcal{F}in$ in $\B$ so that $(\B,\mathcal{F}in)\models T^{fin}$, and such that for all  
$\Theta(x_1,\dots,x_n,w), f_1,\dots,f_n$ there is a $g\in R$ such that if 
$$[[\exists w \Theta(f_1,\dots,f_n,w)]]  \cap \neg [[\exists w (\varphi(w) \wedge \Theta(f_1,\dots,f_n,w))]]\in \mathcal{F}in,$$
then
$$[[\exists w \Theta(f_1,\dots,f_n,w)]]\cap \neg [[\Theta(f_1,\dots,f_n,g)]]\in \mathcal{F}in.$$

This is clearly true in classical products restricted by $\varphi(x)$ (use Axiom of Choice).

In Lemma \ref{lem6+} we need to change "finite" to "in $Fin$", and the proof goes through, getting

\begin{lemma}\label{lem6fin} Suppose $Y_1,\dots,Y_m$ is a partition of $\B$. 
Suppose the sequence
$$<\Theta_0(x_0,\dots,x_k,x_{k+1}),\dots,\Theta_m(x_0,\dots,x_k,x_{k+1})>$$ 
is a partition. Suppose $f_1,\dots,f_k\in R$ and 
$$Y_j\subseteq [[\exists x_{k+1} \Theta_j(f_0,\dots,f_k,x_{k+1})]]$$
for each $j$. Suppose in addition that for each $j$
$$Y_j\setminus [[\exists x_{k+1} \varphi(x_{k+1}) \wedge \Theta_j(f_1,\dots,f_k,x_{k+1})]] \in \mathcal{F}in$$
Then there is a $g$ in $R$ so that
$$Y_j\subseteq [[\Theta_j(f_1,\dots,f_k,g)]]$$
for all $j$.\end{lemma}

We now have Axioms 1-3 and Axiom $4^{fin}$. Note that $\varphi(x)$, the restricting formula, is fixed.

There is one last Axiom 5.

\

{\bf Axiom 5.} $\forall x (Fin([[\neg \varphi(x)]]))$.

\

Call the resulting axiom set $\cA_{\varphi}$, axioms for $\varphi$-restricted products.

\

We have given axioms for pairs $(R, \mathcal{F}in)$, and we shall next prove that we have a Feferman-Vaught type theorem.

\

\section{\bf The Feferman-Vaught Theorem}

\subsection{The Main Theorem}

\begin{thm}\label{main-th} Let $\varphi(\bar x)$ be an $\cL$-formula. Let $R$ a commutative unital ring satisfying the axioms $\cA_{\varphi}$. Then 
for each $\cL_{rings}$-formula $\Theta(x_0,\dots,x_m)$ there is, by an effective procedure, a partition 
$$<\Theta_0(x_0,\dots,x_m),\dots,\Theta_k(x_0,\dots,x_m)>$$ of $\cL_{rings}$-formulas, and an 
$\cL_{Boolean}^{fin}$-formula 
$\psi(y_0,\dots,y_k)$ such that for all $f_1,\dots,f_m$ in $R$
$$R\models \Theta(f_0,\dots,f_m) \Leftrightarrow $$
$$(\B,\mathcal{F}in)\models \psi([[\Theta_0(f_0,\dots,f_m),\dots,\Theta_k(f_0,\dots,f_m)]]).$$\end{thm}
\begin{proof}
In \cite{FV}, there is a standard inductive proof for this by induction on the complexity of $\Theta$ for the case of generalized products. These are the products that are equipped with extra relations making it into a generalized product.  
That proof can be modified to go through for the case of restricted products with respect to a given formula $\varphi$ (which is a substructure of a generalized product). This modification can be made to work in the case of our rings $R$. 

If $\Theta(x_0,\dots,x_m)$ is a quantifier-free formula, then 
we can take the Boolean formula $[[\Theta(x_0,\dots,x_m)]]=1$. Then for all $f_1,\dots,f_m\in R$, 
$$R\models \Theta(f_0,\dots,f_m) \Leftrightarrow \B\models [[\Theta(f_0,\dots,f_m)=1]].$$

Now suppose that $\Theta$ is of the form 
$$\exists x_{m+1} \Theta^*(x_0,\dots,x_m,x_{m+1}),$$
assuming the result known for $\Theta^*$.

Now for any $f_1,\dots,f_m\in R$,
$$R\models \exists x_{m+1} \Theta^*(f_0,\dots,f_m,x_{m+1})$$
if and only if 

$(*_1)$ \ \ \ \ \ \ \ for some $g\in R$
$$R\models \Theta^*(f_0,\dots,f_m,g).$$

By the inductive hypothesis, there is a partition $<\theta'_1,\dots,\theta'_{k'}>$ and a Boolean formula $\Phi'$ (both associated to $\Theta^*$) such that, $(*_1)$ is equivalent to 

$(*_2)$ \ \ \ \ \ \ \ for some $g\in R$
$$R\models \Phi'([[\theta'_0(f_1,\dots,f_m),g]],\dots,[[\theta'_{k'}(f_1,\dots,f_m),g]]).$$

Now we us put $k=k'$, and 
$$\theta_j=\exists x_{k+1} \theta'_j, \ 0\leq j\leq k',$$
and define the following Boolean formula
$$\Phi(z_0,\dots,z_k)=\exists y_1,\dots \exists y_k Part_k(y_0,\dots,y_k) \bigwedge_{j} y_j \leq z_j \Leftrightarrow $$
$$Fin(y_j\setminus [[\exists x_{k+1} \varphi(x_{k+1}) \wedge \theta'_j(f_1,\dots,f_k,x_{k+1})]]) \wedge \Phi'(y_0,\dots,y_k).$$

We show that $(*_2)$ is equivalent to 

$$(*_3)\ \ \ \ \ \  \ \B\models \Phi([[\theta_0(f_1,\dots,f_k)]],\dots,[[\theta_k(f_1,\dots,f_k)]]).$$

Assume $(*_2)$. Define $y_j=[[\theta'_j(f_1,\dots,f_k,g)]]$. Then for each $j$
$$y_j\leq [[\theta_j(f_1,\dots,f_k)]].$$
Since 
$<\theta_1,\dots,\theta_k>$ is a partition, $Part_k(y_0,\dots,y_k)$ holds and $(*_3)$ follows.
 
Conversely, suppose that $(*_3)$ holds. Then there are elements $b_j$ that form a partition of $\B$, and for each $j$ 
we have 
$$b_j \leq [[\exists x_{k+1} \Theta_j(f_0,\dots,f_k,x_{k+1})]],$$
and such that 
$$\B\models Fin(b_j\setminus [[\exists x_{k+1} \varphi(x_{k+1}) \wedge \Theta_j(f_1,\dots,f_k,x_{k+1})]])$$
and
$$\B \models \Phi'(b_0,\dots,b_k).$$

By Lemma \ref{lem6fin} there is $g\in R$ such that for all $j$
$$b_j\subseteq [[\theta'_j(f_1,\dots,f_k,g)]].$$
Since $<b_1,\dots,b_k>$ and $<\theta'_1,\dots,\theta'_k>$ are both partitions, for all $j$ we have
$$b_j=[[\theta'_j(f_1,\dots,f_k,g)]].$$
This proves $(*)_2$.

\end{proof}

\begin{cor} For $R$ as above, with restricting formula $\varphi$, 
$$R \equiv \prod_{e\text{~atom~of~}\B}^{(\varphi)} R_e,$$
the restricted product with respect to $\varphi$.\end{cor}
\begin{proof} Both rings have the same idempotents, the same ideal $\mathcal{F}in$, and same $R_e$ ($e$ an atom), the same $\varphi$, and satisfy the axioms $\cA_{\varphi}$.\end{proof}

\

Note the effectivity and uniformity of Theorem \ref{main-th} in $\varphi$ and all rings satisfying the axioms $\cA_{\varphi}$.

\

\subsection{Ring-Theoretic Definability of $\mathcal{F}in$}

\

\

In \cite{DM-ad} we show that the ideal $\mathcal{F}in$ is $\cL_{rings}$-definable uniformly in all $\A_K$, $K$ a number field. 
In fact there is ring-theoretical definition of $\mathcal{F}in$ for a large class of rings satisfying our axioms 
(and the definition is in a clear sense uniform in $\varphi$).

We have to require the following of $R$ and $\varphi$. 

\

$(\sharp$):~~~ Suppose $e\in \mathcal{F}in$. Then then there are $g,h\in R$ so that 
$$[[e=1]]\subseteq [[gh=1 \wedge \varphi(g) \wedge \neg \varphi(h)]].$$

Note that this is true when $R=\A_{\Q}$.

Now we proceed to an $\cL_{rings}$-definition of $Fin$ assuming $\sharp$.

Suppose first $f\in R$ is an idempotent and $[[f=1]]\notin Fin$. Then there is no $g,h\in R$ with
$$[[f=1]]\subseteq [[gh=1 \wedge \varphi(g) \wedge \neg \varphi(h)]] \subseteq [[\neg \varphi(h)]] \in Fin.$$
On the other hand, suppose $[[f=1]]\in Fin$. Then by $(\sharp)$ there exist $g,h \in R$ with
$$[[f=1]]\subseteq [[gh=1 \wedge \varphi(g) \wedge \neg \varphi(h)]].$$
So we have.

\begin{thm}\label{def-fin} Suppose $R$ satisfies $\sharp$. Then $Fin$ is definable.\end{thm}
\begin{proof} $e=[[e=1]]$, so
$$e\in \mathcal{F}in \Leftrightarrow \exists g \exists h (e\leq [[gh=1 \wedge \varphi(g) \wedge \neg \varphi(h)]]).$$ 
\end{proof}

\bibliographystyle{acm}
\bibliography{bibadeles}

\end{document}